\documentclass[12pt,a4paper]{amsart}
\usepackage{amsmath,amssymb,amsfonts,amsthm}
\usepackage{tikz}
\usepackage{float}
\usetikzlibrary{arrows.meta,positioning,calc}
\usepackage[colorlinks=true,linkcolor=red,citecolor=blue,urlcolor=blue]{hyperref}

\numberwithin{equation}{section}

\newtheorem{theorem}{Theorem}[section]
\newtheorem{lemma}[theorem]{Lemma}
\newtheorem{proposition}[theorem]{Proposition}
\theoremstyle{definition}
\newtheorem{definition}[theorem]{Definition}
\theoremstyle{remark}
\newtheorem{remark}[theorem]{Remark}
\newtheorem*{conjecture}{Conjecture}

\newcommand{\Z}{\mathbb{Z}}
\newcommand{\R}{\mathbb{R}}
\newcommand{\C}{\mathbb{C}}
\newcommand{\N}{\mathbb{N}}
\newcommand{\Q}{\mathbb{Q}}
\newcommand{\calA}{\mathcal{A}}

\DeclareMathOperator{\supp}{supp}

\title[TILES AND WEAK TILES IN $\Z_{pq}$]{TILES AND WEAK TILES IN $\Z_{pq}$}

\author{Mamateli Kadir}
\author{Kaibo Fan}

\thanks{M. Kadir is supported by NSF of China (Grant No. 12361015) and by NSF of Xinjiang Uygur Autonomous Region, P. R. China (Grant No. 2025D01A09).}

\begin{document}

\begin{abstract}
This paper investigates the relationship between tiles and weak tiles in the context of finite cyclic group $\Z_{pq}$. We prove that weak tiles and translational tiles are equivalent in this group. Our proof employs Fourier analysis, Delsarte parameters, and the Coven-Meyerowitz conditions.
\end{abstract}

\maketitle

\section{Introduction}

\subsection{Translational tiling and spectral set.}
The study of tiling, the covering of space by congruent copies of a given shape, stretches back to the dawn of civilization. In modern mathematics, we usually study tiling in the context of having a finite set of ''shapes'' called prototiles and using congruent copies of these prototiles to cover the whole Euclidean space without overlapping. Translational tiling is one such example, in which only the translations of the prototiles are used to tile the space.

Let $\Omega \subset \R^d$ be a bounded measurable set with positive finite Lebesgue measure. If there exists a set of translation vectors $T \subset \R^d$ such that the family of translated copies $\{\Omega+t\}_{t\in T}$ covers the whole space $\R^d$ almost everywhere without overlap, i.e.,
\[
    \sum_{t\in T} \mathbf{1}_{\Omega}(x-t)=1, \quad \text{a.e. } x\in \R^d.
\]
Then the set $\Omega \subset \R^d$ is called a \textbf{tile}, and $T$ is a \textbf{tiling set}.

On the other hand, if there exists a countable set $\Lambda \subset \R^d$ such that exponential function system
\[
    E(\Lambda)=\{e^{2\pi i \langle \lambda,x\rangle}:\lambda\in\Lambda\}
\]
is a complete orthogonal basis of $L^2(\Omega)$, then $\Omega$ is called a \textbf{spectral set}, and $\Lambda$ is a \textbf{spectrum} of $\Omega$.

The translational tile is closely related to the spectral set conjecture, also known as the Fuglede's conjecture, a fundamental question at the intersection of geometry and harmonic analysis. it was first posed by Fuglede in 1974 \cite{Fuglede}, motivated by a problem of Segal concerning the commutativity of certain partial differential operators.

\begin{conjecture}
A Borel set $\Omega \subset \R^d$ of positive and finite Lebesgue measure is a spectral set if and only if it is a translational tile.
\end{conjecture}

This conjecture has spurred extensive research, revealing deep connections to combinatorics, number theory, and harmonic analysis.

While progress was made in special cases, including convex planar domains \cite{IosevichKatzTao}, convex polytopes in $\R^3$ \cite{GreenfeldLev} and unions of two intervals on the real line \cite{Laba}, the conjecture was ultimately disproved in its full generality in dimensions 3 and above through a series of counterexamples for both directions \cite{Tao,FarkasMatolcsiMora,KolountzakisMatolcsiHadamard,KolountzakisMatolcsiTiles,Matolcsi}, Despite these, the conjecture's validity in dimensions 1 and 2 remains open, and a landmark result established its truth for all convex bodies in every dimension \cite{LevMatolcsi}.

In particular, since high-dimensional counterexamples in Euclidean spaces arise from finite groups, the conjecture on finite abelian groups has attracted considerable interest \cite{Malikiosis,MalikiosisKolountzakis,ShiPQR,ShiEqui,Somlai,Zhang}.

\subsection{Weak tiling and pd-tiling.}
A pivotal concept in the proof for convex bodies \cite{LevMatolcsi} was that of \textbf{weak tiling}, which relaxes the stringent conditions of classical translational tiling. A set $A\subseteq \R^d$ weakly tiles another measurable set $B\subseteq \R^d$, if there exists a positive, locally finite measure $\nu$ on $\R^d$ such that \(\nu(\{0\})=1\) and $\mathbf{1}_A*\nu=\mathbf{1}_B$ a.e.. This analytic formulation provided the crucial bridge to connect spectrality to a geometric property.

Subsequently, Kiss, Matolcsi, Matolcsi, and Somlai \cite{KissMatolcsiMatolcsiSomlai} adapted this notion to finite abelian groups $G$.

\begin{definition}
A set $A$ weakly tiles $G$, if there exists a nonnegative function $f:G\to\R_+$ such that \(f(0)=1\) and $\mathbf{1}_A*f=\mathbf{1}_G$.
\end{definition}

To align with the harmonic analysis framework, they introduced the more refined concept of a \textbf{positive-definite weak tile} (or \textbf{pd-tile}), where the function $f$ is additionally required to be positive definite and normalized $f(0)=1$.

\begin{definition}
A set $A\subset G$ is said to be a \textbf{positive-definite weak tile} of $G$ (abbreviated as pd-tile), if there exists a function $f:G\to\R$ such that
\[
    f(0)=1, \quad f\geq 0, \quad \widehat{f}\geq 0, \quad \mathbf{1}_A*f=\mathbf{1}_G.
\]
\end{definition}

This definition is powerful because both classical tiles and spectral sets are known to be pd-tiles\cite{KissMatolcsiMatolcsiSomlai}. Consequently, establishing the ``spectral$\Rightarrow$ tile'' direction of Fuglede's conjecture in a finite abelian group $G$ is equivalent to proving that every pd-tile in $G$ is a classical tile. The first examples of sets that are \textbf{pd-tiles} but neither spectral nor classical tiles,termed ``lonely weak tiles''were recently constructed by Kiss, Londner, Matolcsi, and Somlai \cite{KissLondnerMatolcsiSomlaiLonely}, demonstrating the non-triviality of this relaxation.

\subsection{The Coven--Meyerowitz conditions.}
Parallel to the study of Fugledes conjecture, the theory of translational tilings of the integers has its own central open problem: \textbf{the Coven--Meyerowitz conjecture}. For a finite set $A\subseteq \Z$ tiles $\Z$, it is well known that tiling of $\Z$ by translates of $A$ is equivalent, after periodization, to a tiling $A\oplus B=\Z_M$ in a cyclic group for some period $M$. Coven--Meyerowitz \cite{CovenMeyerowitz} formulated two algebraic conditions on the cyclotomic divisibility of the mask polynomial $A(X)=\sum_{a\in A}X^a$:
\begin{align*}
	\text{(T1)}\qquad 
	& A(1)=\prod_{s\in S_A}\Phi_s(1), 
	\quad \text{where}\quad 
	S_A=\{p^\alpha:\Phi_{p^\alpha}(X)\mid A(X)\}.\\
	\text{(T2)}\qquad 
	& \text{if } s_1,\ldots,s_k\in S_A 
	\text{ are powers of distinct primes, then}\\
	& \Phi_{s_1\cdots s_k}(X)\mid A(X).
\end{align*}

They proved \cite{CovenMeyerowitz} that $(T1)+(T2)$ are sufficient for tiling, and that $(T1)$ is necessary. The necessity of (T2) for all tilings, \textbf{the Coven--Meyerowitz conjecture}, remains a major open problem, having been verified only in special cases. If $|A|$ has at most two prime factors, then necessity of $(T2)$ is true. Thus, a finite set $A\subseteq\Z$ with at most two prime factors, $A$ tiles $\Z$ if and only if $A$ satisfies condition (T1) and (T2).

The preceding work raises a fundamental question: on which groups does the equivalence between \textbf{pd-tiling} and translational tiling hold? Currently, this equivalence is known only for finite groups $\Z_p$, $\Z_p^2$, and $\Z_{p^n}$. The main result of this paper significantly extends this list by proving the equivalence for the cyclic group $\Z_{pq}$, where $p$ and $q$ are distinct primes.

A relationship among the different concepts in a cyclic group is given in the following Figure \ref{fig:relations}.

\begin{theorem}\label{thm:main}
Let $p,q$ be distinct primes. For any subset $A\subset \Z_{pq}$, if $A$ is a positive-definite weak tile, then $A$ is a translational tile, i.e., there exists $B\subset \Z_{pq}$ such that $A\oplus B=\Z_{pq}$.
\end{theorem}

\newpage
\begin{figure}
	\centering
	\definecolor{deepyellow}{RGB}{210, 160, 0} 
	\begin{tikzpicture}[
		scale=1,
		>={Stealth[length=3mm, width=2mm]},
		thick,
		box/.style={draw, align=center, minimum height=1.0cm, thick},
		wavedStyle/.style={decorate, decoration={snake, amplitude=1mm, segment length=5mm}}
		]
		\path[use as bounding box] (-8.5, -2.5) rectangle (8.5, 6.5);
		
		\definecolor{deepyellow}{RGB}{210, 160, 0}
		\node[box, minimum width=3.0cm, minimum height=1.1cm] (cm) at (0,5.2) {C-M\\ $(T1)+(T2)$};
		\node[box, minimum width=2.6cm] (tile) at (-4.5,3.2) {Tile};
		\node[box, minimum width=3.0cm] (spec) at (4.5,3.2) {Spectral};
		\node[box, minimum width=3.2cm] (weak) at (0,1.5) {Weak tile};
		
		\node[box, minimum width=2.8cm] (pd) at (0,-1.4) {Pd-tile};
		\draw[->] (cm.210) -- (tile.70);
		\draw[->, deepyellow] (tile.110) to[bend left=20] node[midway, above left, inner sep=2pt] {\large ?} (cm.180);
		
		\draw[->] (cm.330) -- (spec.110);
		\draw[->, deepyellow] (spec.70) to[bend right=20] node[midway, above right, inner sep=2pt] {\large ?} (cm.0);
		
		\draw[->] (tile.290) -- (weak.150);
		\draw[->, deepyellow] (weak.180) to[bend left=25] node[midway, below left, inner sep=2pt] {\Large ?} (tile.220);
		
		\draw[->] (spec.250) -- (weak.30);
		\draw[->, deepyellow] (weak.0) to[bend right=25] node[midway, below right, inner sep=2pt] {\large  ?} (spec.320);
		\draw[->] (pd.90) -- (weak.270);
		
		\draw[->, deepyellow] (weak.240) to node[left] {?} (pd.120);
		
	\end{tikzpicture}
\caption{The relationship between tiles, spectral sets, weak tiles, pd-tiles, and C-M conditions in cyclic groups.}
\label{fig:relations}
\end{figure}
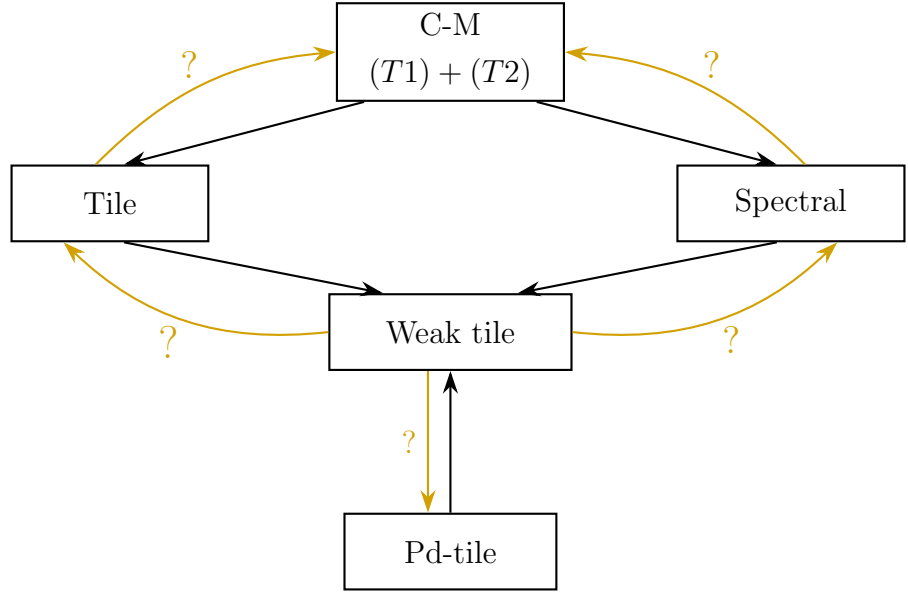

The remainder of the paper is organized as follows. In Section \ref{sec:prelim}, we collect the necessary preliminaries, including the Fourier transform on finite abelian groups, mask polynomials and cyclotomic divisibility, step functions, and the Delsarte linear programming framework. Section \ref{sec:equivalence} presents the proof of the main result (Theorem \ref{thm:main}). The proof proceeds by classifying all possible Fourier supports of subsets of $\Z_{pq}$, employing Delsarte parameters to derive tight bounds on the size of a pd-tile, and finally verifying the Coven--Meyerowitz conditions to conclude that the set is a translational tile. We conclude the paper with a discussion of open problems, including possible extensions to more general cyclic groups and the structural decomposition of weak tiling measures.

\section{Preliminaries}\label{sec:prelim}

In this section we collect the necessary background material from harmonic analysis on finite abelian groups, the theory of mask polynomials and cyclotomic divisibility, and the Delsarte linear programming framework. These tools will be essential for the proof of our main theorem.

\subsection{Finite abelian groups and characters.}
Let $G$ be a finite abelian group. A \textbf{character} of $G$ is a group homomorphism $\chi:G\to\C^\times$, i.e., a map satisfying
\[
    \chi(x+y)=\chi(x)\chi(y), \quad \forall x,y\in G.
\]

The set of all characters of $G$, denoted by $\widehat{G}$, forms a group under pointwise multiplication, called the dual group of $G$. It is a standard fact that $\widehat{G}$ is isomorphic to $G$ itself, though not canonically.

Every finite abelian group $G$ admits a decomposition
\[
    \Z_{n_1}\times \Z_{n_2}\times \cdots \times \Z_{n_s},
\]
with $n_1\mid n_2\mid\cdots\mid n_s$. For an element $g=(g_1,g_2,\ldots,g_s)\in G$, we define the associated character $\chi_g:G\to\C^\times$ by
\[
    \chi_g(x_1,\ldots,x_s)=e^{2\pi i\sum_{j=1}^{s}\frac{x_jg_j}{n_j}},
\]
where $x=(x_1,\ldots,x_s)\in G$. This identification gives an explicit isomorphism $G\cong \widehat{G}$.

For a product of two finite abelian groups $G_1\times G_2$, every character decomposes as
\[
    \chi_{g_1,g_2}(x_1,x_2)=\chi_{g_1}(x_1)\cdot\chi_{g_2}(x_2),
\]
where, $g_1\in G_1$, $g_2\in G_2$ and $x_j\in G_j$. In particular, $\widehat{G_1\times G_2}\cong \widehat{G}_1\times\widehat{G}_2$.

\subsection{Fourier transform on finite abelian groups.}
For a function $f:G\to\C$, the \textbf{Fourier transform} $\widehat{f}:\widehat{G}\to\C$ is defined by
\[
    \widehat{f}(\chi)=\sum_{x\in G}f(x)\overline{\chi(x)}, \quad \chi\in\widehat{G}.
\]
With the identification $G\cong\widehat{G}$, we shall often write $\widehat{f}(\xi)$ for $\xi\in G$, where the character is understood to be $\chi_{\xi}(x)=e^{2\pi i\frac{\langle \xi,x\rangle}{M}}$ in the cyclic case.

The Fourier inversion transform is defined as
\[
    f(x)=\frac{1}{|G|}\sum_{\chi\in\widehat{G}}\widehat{f}(\chi)\chi(x), \quad x\in G.
\]

Convolution of two functions $f,g:G\to\C$, is defined by
\[
    (f*g)(x)=\sum_{y\in G}f(x-y)g(y),
\]
and satisfies the fundamental identity
\[
    \widehat{f*g}(\chi)=\widehat{f}(\chi)\widehat{g}(\chi).
\]

When $G=\Z_M$, we identify $\Z_M$ with the set $\{0,1,,M-1\}$. The Fourier transform then takes the familiar form

\[
    \widehat{f}(\xi)=\sum_{z=0}^{M-1}f(z)e^{-2\pi iz\xi/M}, \quad \xi=0,1,\cdots,M-1.
\]

\subsection{Mask polynomials, step functions and the C-M condition.}
Let $A\subset \Z$ be a finite set. If there exists a set $T\subset \Z$ such that the family of translates $\{A+t:t\in T\}$ is pairwise disjoint and
\[
    \bigcup_{t\in T}(A+t)=\Z,
\]
then we say that $A$ tiles $\Z$ by translations.

Equivalently, if there exists $T\subset\Z$ such that every $n\in\Z$ can be uniquely represented as
\[
    n=a+t, \quad a\in A,\ t\in T,
\]
we write $A\oplus T=\Z$.

A classical theorem of Newman \cite{Newman} states that any tiling complement $T$ must be periodic, i.e., there exists $M\in\N$ and a finite set $B\subset\Z$ such that $T=B\oplus M\Z$. Consequently, $|A||B|=M$ and $A\oplus B=\Z_M$, where the addition is taken modulo $M$.

Let $M=\prod_{i=1}^K p_i^{n_i}$ be the prime factorization of $M$. Let $f:\Z_M\to\C$ be any function. For $\Z_M\cong\{0,\ldots,M-1\}$, we define the mask polynomial of $f$ as
\begin{equation}\label{eq:maskpoly}
    F(X):=\sum_{z\in\Z_M}f(z)X^z.
\end{equation}
Recall that the Fourier transform on $\Z_M$ is given by
\[
    \widehat{f}(\xi)=\sum_{z\in\Z_M}f(z)e^{-2\pi iz\xi/M}, \quad (\xi=0,1,\ldots,M-1).
\]
With $\eta=e^{-2\pi i/M}$, we have the fundamental relation
\[
    F(\eta^\xi)=\widehat{f}(\xi).
\]
For nonnegative $f$, the total mass of $f$ is denoted
\begin{equation}\label{eq:totalmass}
    |F|:=F(1)=\sum_{z\in\Z_M}f(z)=\widehat{f}(0).
\end{equation}
The $n$-th cyclotomic polynomial $\Phi_n(X)$ is defined as
\[
    \Phi_n(X)=\prod_{\substack{1\leq k\leq n\\ (k,n)=1}}(X-e^{2\pi i k/n}),
\]
which is the minimal polynomial over $\Q$ of primitive $n$-th root of unity.

For each divisor $m\mid M$, we define the \textbf{step class}
\[
    R_m:=\{z\in\Z_M:(z,M)=m\}.
\]
Note that $R_M=\{0\}$.

\begin{definition}
A function $f$ is called a \textbf{step function} if it is constant on each step class $R_m$. We denote the vector space of step functions by $\calA(M)$.
\end{definition}

A fundamental result, proved in \cite{KissLondnerMatolcsiSomlaiFunctional}, is that $\calA(M)$ is invariant under the Fourier transform.

For a set $A\subset\Z_M$, its mask polynomial, taking $f=\mathbf{1}_A$ in (\ref{eq:maskpoly}), is defined as
\[
    A(X):=\sum_{a\in A}X^a.
\]
Moreover, from (\ref{eq:totalmass}) the total mass of $\mathbf{1}_A$ coincides with the cardinality $|A|$ of $A$.

The following lemma is standard and will be used repeatedly.

\begin{lemma}\label{lem:step-support}
Let $A\subset \Z_M$. For any divisor $m\mid M$, the Fourier transform $\widehat{\mathbf{1}}_A$ is either identically zero on the step class $R_m$ or nowhere zero. Consequently, support of $\widehat{\mathbf{1}}_A$ is a union of certain step classes.
\end{lemma}

\begin{proof}
By the definition of the discrete Fourier transform,
\[
    \widehat{\mathbf{1}}_A(\xi)=\sum_{a\in A}\eta^{a\xi}=A(\eta^\xi), \quad \xi\in\Z_M.
\]
Fix $m\mid M$ and take $\xi\in R_m$, that means $(\xi,M)=m$. Let $n=\frac{M}{m}$, since $(\xi,M)=m$, we can write $\xi=mu$ with $(u,n)=1$. Then
\[
    \eta^\xi=e^{-2\pi i\xi/M}=e^{-2\pi iu/n},
\]
so $\eta^\xi$ is a primitive $n$-th root of unity. When $\xi$ runs over $R_m$, $\eta^\xi$ runs over all primitive $n$-th roots of unity.

Suppose there exists some $\xi_0\in R_m$ such that $\widehat{\mathbf{1}}_A(\xi_0)=0$, then $A(\eta^{\xi_0})=0$. Since $\eta^{\xi_0}$ is a primitive $n$-th root of unity and
\[
    A(X)\in\Z[X]\subset\Q[X],
\]
we have
\[
    \Phi_n(X)\mid A(X).
\]
Hence $A(X)$ vanishes at all primitive $n$-th roots of unity. Consequently, for every $\xi\in R_m$,
\[
    \widehat{\mathbf{1}}_A(\xi)=A(\eta^\xi)=0.
\]

Thus, if $\widehat{\mathbf{1}}_A$ has a zero in $R_m$, it is identically zero on the whole $R_m$. Therefore, for any $m\mid M$, $\widehat{\mathbf{1}}_A$ is either identically zero on $R_m$ or has no zero there. Hence the support of $\widehat{\mathbf{1}}_A$ is a union of step classes.
\end{proof}

Let $\Phi_n(X)$ denote the $n$-th cyclotomic polynomial, the unique monic irreducible polynomial whose roots are the primitive $n$-th roots of unity. The tiling condition for two sets $A,B\subset \Z_M$ can be expressed neatly in terms of cyclotomic divisibility. Specifically, assuming $0\in A\cap B$, we have
\[
    A\oplus B=\Z_M
\]
if and only if
\[
    |A|\,|B|=M,
\]
and for every divisor $m\mid M$ with $m\neq 1$,
\[
    \Phi_m(X)\mid A(X) \quad \text{or}\quad \Phi_m(X)\mid B(X).
\]
Equivalently, in Fourier language,
\[
    |A|\,|B|=M, \quad \text{and}\quad \widehat{\mathbf{1}}_A(\xi)\cdot\widehat{\mathbf{1}}_B(\xi)=0 \quad \text{for all } \xi\neq 0.
\]

Note that for any $m\mid M$,
\[
    \Phi_m(X)\mid A(X) \Longleftrightarrow A(\eta^\xi)=0, \quad \text{for all } \xi\in R_{M/m}
\]
which can also be expressed as
\begin{equation}\label{eq:cyclo}
    \Phi_m(X)\mid A(X) \Longleftrightarrow \widehat{\mathbf{1}}_A(\xi)=0, \quad \text{for all } \xi\in R_{M/m}.
\end{equation}

For a finite set $A\subset\Z$, define
\[
    S_A:=\{p^\alpha:p^\alpha \text{ is a prime power and } \Phi_{p^\alpha}(X)\mid A(X)\}.
\]
The Coven--Meyerowitz conditions are:
\begin{itemize}
\item[(T1)] $A(1)=\prod_{s\in S_A}\Phi_s(1)$;
\item[(T2)] If $s_1,\ldots,s_k\in S_A$ are powers of distinct primes, then $\Phi_{s_1\cdots s_k}(X)\mid A(X)$.
\end{itemize}
It is known \cite{CovenMeyerowitz} that (T1)+(T2) are sufficient for $A$ to tile $\Z$, and that tiling implies (T1). Moreover, if $|A|$ has at most two distinct prime factors, then tiling also implies (T2). In the special case $M=pq$ with $p,q$ distinct primes, these conditions provide a complete characterization of tiles.

\section{Equivalence of pd-tiling and translational tiling on $\Z_{pq}$}\label{sec:equivalence}
Every translational tile is a positive-definite weak tile.

\begin{lemma}\label{lem:tile-implies-pd}
	If $A\oplus B=G$, then $A$ is a positive-definite weak tile.
\end{lemma}

\begin{proof}
	Translate $B$ if necessary so that $0\in B$, and define
	\[
	f=\frac{1}{|B|}\mathbf{1}_B*\mathbf{1}_{-B}.
	\]
	Then $f(0)=1$ and $f\ge 0$. Moreover,
	\[
	\widehat f
	=
	\frac{1}{|B|}
	\widehat{\mathbf{1}_B}\,
	\widehat{\mathbf{1}_{-B}}
	=
	\frac{1}{|B|}
	\left|\widehat{\mathbf{1}_B}\right|^2
	\ge 0.
	\]
	Finally, since $A\oplus B=G$, we have
	\[
	\mathbf{1}_A*\mathbf{1}_B=\mathbf{1}_G.
	\]
	Hence
	\[
	\mathbf{1}_A*f
	=
	\frac{1}{|B|}
	(\mathbf{1}_A*\mathbf{1}_B)*\mathbf{1}_{-B}
	=
	\frac{1}{|B|}
	\mathbf{1}_G*\mathbf{1}_{-B}
	=
	\mathbf{1}_G.
	\]
	Therefore $A$ is a positive-definite weak tile.
\end{proof}

\subsection{Delsarte parameters.}
We now recall the Delsarte linear programming framework, which will be central to our proof.

\begin{definition}
For a collection of step classes $H=\bigcup_{i=1}^r R_{m_i}$ with $0\in H$, the \textbf{standard complement} of $H$ is defined by
\[
    H':=(\Z_M\setminus H)\cup\{0\}.
\]
For a parameter $\delta>0$, set
\[
D^+(H)=\max\left\{\sum_{z\in\Z_M}h(z):\begin{array}{l}
h\in\calA(M),\ h(0)=1,\ \widehat{h}\geq 0,\\
h\geq 0 \text{ on } H,\ h=0 \text{ on } H^c
\end{array}\right\},
\]
\[
D^-(H)=\max\left\{\sum_{z\in\Z_M}h(z):\begin{array}{l}
h\in\calA(M),\ h(0)=1,\ \widehat{h}\geq 0,\\
h\leq 0 \text{ on } H^c
\end{array}\right\},
\]
and
\[
D^{\delta+}(H)=\max\left\{\sum_{z\in\Z_M}h(z):\begin{array}{l}
h\in\calA(M),\ h(0)=1,\ \widehat{h}\geq 0,\\
h\geq \delta \text{ on } H,\ h=0 \text{ on } H^c
\end{array}\right\}.
\]
These parameters satisfy the hierarchy
\[
    D^{\delta+}(H)\leq D^+(H)\leq D^-(H),
\]
and the duality relation \cite{KissLondnerMatolcsiSomlaiFunctional}
\begin{equation}\label{eq:duality}
    D^+(H)D^-(H')=M.
\end{equation}
\end{definition}

The following proposition will be crucial.

\begin{proposition}[{\cite[Proposition 3.2]{KissLondnerMatolcsiSomlaiFunctional}}]\label{prop:delsarte}
Let $H\subset\Z_M$ be a union of step classes with $0\in H$. If there exists a tiling $A\oplus B=\Z_M$ such that $H=\supp |\widehat{\mathbf{1}}_A|^2$, then
\[
    D^+(H)=D^-(H)=|B|, \quad \text{and}\quad D^+(H')=D^-(H')=|A|.
\]
\end{proposition}

From Proposition \ref{prop:delsarte}, if there exists a tiling $D\oplus C=\Z_{pq}$ with $H_D:=\supp |\widehat{\mathbf{1}}_D|^2$, then
\[
    D^+(H')=|D|.
\]
Therefore, on $\Z_{pq}$, there are exactly six possible supports $H_D:=\supp |\widehat{\mathbf{1}}_D|^2$, arising from classical tilings $D\oplus C=\Z_{pq}$.

We list them below for reference.
\begin{enumerate}
\item $D=\{0,1,\ldots,p-1\}$, \quad $C=p\Z_{pq}$. Then
\[
    \widehat{\mathbf{1}}_D(\xi)=\sum_{d=0}^{p-1}e^{-2\pi id\xi/(pq)}=\frac{1-e^{-2\pi ip\xi/(pq)}}{1-e^{-2\pi i\xi/(pq)}}.
\]
Thus $\widehat{\mathbf{1}}_D(\xi)=0\Longleftrightarrow q\mid \xi$ and $pq\nmid \xi$, i.e. exactly on $R_q$. Hence
\[
    H_D=\{0\}\cup R_1\cup R_p.
\]

\item Similarly, swapping $p$ and $q$, $D=\{0,1,\ldots,q-1\}$, $C=q\Z_{pq}$, then
\[
    H_D=\{0\}\cup R_1\cup R_q.
\]

\item Let $D=p\Z_{pq}=\{0,p,2p,\ldots,(q-1)p\}$, \quad $C=\{0,1,\ldots,p-1\}$. Then
\[
    \widehat{\mathbf{1}}_D(\xi)=\sum_{t=0}^{q-1}e^{-2\pi i(pt)\xi/(pq)}=\sum_{t=0}^{q-1}e^{-2\pi it\xi/q}.
\]
This is a sum of $q$th roots of unity. $\widehat{\mathbf{1}}_D(\xi)\neq 0$ iff $\xi\equiv 0\ (\mathrm{mod}\ q)$. Hence
\[
    H_D=\{0\}\cup R_q.
\]

\item $D=q\Z_{pq}$, $C=\{0,1,\ldots,q-1\}$, then
\[
    H_D=\{0\}\cup R_p.
\]

\item $D=\Z_{pq}$, \quad $C=\{0\}$. Then $\widehat{\mathbf{1}}_D$ is zero for all $\xi\neq 0$, so
\[
    H_D=\{0\}.
\]

\item $D=\{0\}$, \quad $C=\Z_{pq}$. Then $\widehat{\mathbf{1}}_D\equiv 1$, so
\[
    H_D=\Z_{pq}.
\]
\end{enumerate}
By Proposition \ref{prop:delsarte}, for all the above cases we have
\[
    D^+(H_D')=|D|.
\]

The following lemma rules out one specific Fourier support that would otherwise be exceptional.

\begin{lemma}\label{lem:exceptional}
Let $p,q$ be distinct primes and $A\subset \Z_{pq}$ nonempty. Then
\[
    \supp(|\widehat{\mathbf{1}}_A|^2)\neq \{0\}\cup R_p\cup R_q.
\]
\end{lemma}

\begin{proof}
Because $\supp(|\widehat{\mathbf{1}}_A|^2)=\supp(\widehat{\mathbf{1}}_A)$. Hence it suffices to prove
\[
    \supp(\widehat{\mathbf{1}}_A)\neq \{0\}\cup R_p\cup R_q.
\]

By the Chinese remainder theorem, $\Z_{pq}\cong\Z_p\times\Z_q$. Then we regard $A$ as a subset of $\Z_p\times\Z_q$ and write
\[
    a_{ij}:=\mathbf{1}_A(i,j)\in\{0,1\}, \quad i\in\Z_p,\ j\in\Z_q.
\]
Let
\[
    r_i:=\sum_{j\in\Z_q}a_{ij}, \quad c_j:=\sum_{i\in\Z_p}a_{ij}
\]
be the row and column sums.

Under the identification $\Z_{pq}\cong\Z_p\times\Z_q$, the step classes are
\[
    R_p=\{(0,v):v\neq 0\}, \quad R_q=\{(u,0):u\neq 0\}, \quad R_1=\{(u,v):u\neq 0, v\neq 0\}.
\]
Thus the assumption
\[
    \supp(\widehat{\mathbf{1}}_A)=\{0\}\cup R_p\cup R_q
\]
is equivalent to
\[
    \widehat{\mathbf{1}}_A(u,v)=0, \quad \text{for all } u\neq 0,\ v\neq 0.
\]
On the other hand, by the Fourier inversion formula on $\Z_p\times\Z_q$,
\[
    a_{ij}=\frac{1}{pq}\sum_{u\in\Z_p}\sum_{v\in\Z_q}\widehat{\mathbf{1}}_A(u,v)e^{2\pi i(\frac{ui}{p}+\frac{vj}{q})}.
\]
Since $\widehat{\mathbf{1}}_A(u,v)=0$ for $u\neq 0$, $v\neq 0$, this reduces to
\[
    a_{ij}=\frac{1}{pq}\left(\sum_{u\in\Z_p}\widehat{\mathbf{1}}_A(u,0)e^{2\pi iui/p}+\sum_{v\in\Z_q}\widehat{\mathbf{1}}_A(0,v)e^{2\pi ivj/q}-\widehat{\mathbf{1}}_A(0,0)\right).
\]
Using,
\[
    \widehat{\mathbf{1}}_A(u,0)=\sum_{i\in\Z_p}r_ie^{-2\pi iui/p}, \quad \widehat{\mathbf{1}}_A(0,v)=\sum_{j\in\Z_q}c_je^{-2\pi ivj/q},
\]
and $\widehat{\mathbf{1}}_A(0,0)=|A|$. Applying one dimensional Fourier inversion to the sequences $\{r_i\}$ and $\{c_j\}$ gives
\[
    a_{ij}=\frac{r_i}{q}+\frac{c_j}{p}-\frac{|A|}{pq}.
\]
Thus for any $i,i'\in\Z_p$ and $j,j'\in\Z_q$,
\begin{equation}\label{eq:diffs}
    a_{ij}-a_{ij'}=\frac{c_j-c_{j'}}{p}, \quad a_{ij}-a_{i'j}=\frac{r_i-r_{i'}}{q}.
\end{equation}
We claim that no row of the matrix $(a_{ij})$ can contain both $0$ and $1$. Suppose, for contradiction, that there exists some $i$ and $j\neq j'$ such that
\[
    a_{ij}=1, \quad a_{ij'}=0.
\]

From (\ref{eq:diffs}), $1=a_{ij}-a_{ij'}=\frac{c_j-c_{j'}}{p}$, hence $c_j-c_{j'}=p$. Since $0\leq c_j,c_{j'}\leq p$, we must have $c_j=p$ and $c_{j'}=0$. Thus column $j$ is all 1's and column $j'$ is all 0's.

Now for any column $k$, using (\ref{eq:diffs}) with the all-zero column $j'$, we get
\[
    a_{ik}=\frac{c_k}{p} \quad \text{for all } i.
\]
Since $a_{ik}\in\{0,1\}$, we have $\frac{c_k}{p}\in\{0,1\}$, i.e., $c_k\in\{0,p\}$. Hence every column is constant, so $A$ is a union of whole columns. In other words, there exist $b_j\in\{0,1\}$ such that
\[
    a_{ij}=b_j, \quad \text{for all } i,j.
\]
Then
\[
    \widehat{\mathbf{1}}_A(u,v)=\left(\sum_{i\in\Z_p}e^{-2\pi iui/p}\right)\left(\sum_{j\in\Z_q}b_je^{-2\pi ivj/q}\right).
\]
For $u\neq 0$, the first factor is zero, so $\widehat{\mathbf{1}}_A(u,v)=0$ for all $u\neq 0$ and all $v$. Consequently,
\[
    \supp(\widehat{\mathbf{1}}_A)\subset \{0\}\cup R_p,
\]
contradicting the assumption that $\supp(\widehat{\mathbf{1}}_A)=\{0\}\cup R_p\cup R_q$. Hence no row can contain both 0 and 1.

By symmetry, exchanging rows and columns, and swapping $p$ and $q$, the same argument shows that no column can contain both 0 and 1. Therefore every row is constant and every column is constant. If the matrix $(a_{ij})$ contains both 0 and 1, pick indices $i_0,j_0,i_1,j_1$ such that $a_{i_0j_0}=1$ and $a_{i_1j_1}=0$. Since row $i_0$ is constant, $a_{i_0j_1}=1$; but column $j_1$ is constant, so $a_{i_0j_1}=0$, a contradiction. Hence the matrix is constant, so $A=\emptyset$ or $A=\Z_{pq}$.

Since $A$ is nonempty, we must have $A=\Z_{pq}$. But then $\widehat{\mathbf{1}}_A(\xi)=0$ for all $\xi\neq 0$, so $\supp(\widehat{\mathbf{1}}_A)=\{0\}$, again a contradiction.

Thus
\[
    \supp(|\widehat{\mathbf{1}}_A|^2)\neq \{0\}\cup R_p\cup R_q.
\]
\end{proof}

\subsection{Proof of the Theorem \ref{thm:main}.}
Now we give a proof of Theorem \ref{thm:main}.

\begin{proof}[Proof of the Theorem \ref{thm:main}]
Let $f:\Z_{pq}\to\R$ be a positive-definite function satisfying
\[
    f(0)=1, \quad f\geq 0, \quad \widehat{f}\geq 0, \quad \mathbf{1}_A*f=\mathbf{1}_{\Z_{pq}}.
\]
Taking Fourier transform of $\mathbf{1}_A*f=\mathbf{1}_{\Z_M}$ gives
\[
    \widehat{\mathbf{1}}_A(\xi)\widehat{f}(\xi)=pq\delta_0(\xi) \quad (\xi\in\Z_{pq}).
\]
Thus
\[
    \supp(\widehat{\mathbf{1}}_A)\cap\supp(\widehat{f})=\{0\}.
\]
Set
\[
    H:=\supp(|\widehat{\mathbf{1}}_A|^2)\subset\Z_{pq}, \quad \text{and}\quad H'=(\Z_{pq}\setminus H)\cup\{0\}.
\]
\textbf{Step 1:} Determine all possible $H$. By Lemma \ref{lem:step-support}, $H$ is a union of step classes, and the nontrivial step classes are $R_1,R_p,R_q$ ($R_{pq}=\{0\}$). There are 8 possible unions containing $\{0\}$. We first rule out two of them.
\begin{enumerate}
\item[(i)] If $H=\{0\}\cup R_1$, then $\widehat{\mathbf{1}}_A$ vanishes on both $R_p$ and $R_q$. By (\ref{eq:cyclo}), this implies
\[
    \Phi_p(X)\mid A(X), \quad \Phi_q(X)\mid A(X).
\]
Hence $pq=\Phi_p(1)\Phi_q(1)\mid A(1)=|A|$. Since $0<|A|\leq pq$, we must have $|A|=pq$, i.e., $A=\Z_{pq}$. But then $\widehat{\mathbf{1}}_A$ is zero for all $\xi\neq 0$, so $H=\{0\}$, contradiction.
\item[(ii)] If $H=\{0\}\cup R_p\cup R_q$, this is impossible by Lemma \ref{lem:exceptional}.
\end{enumerate}
Therefore, $H$ can only be one of the following six possibilities.
\[
    \{0\}, \ \Z_{pq}, \ \{0\}\cup R_1\cup R_p, \ \{0\}\cup R_1\cup R_q, \ \{0\}\cup R_p, \ \{0\}\cup R_q.
\]

\textbf{Step 2:} Determine $D^+(H')$ using Proposition \ref{prop:delsarte} and the model tilings.
\begin{itemize}
\item If $H=\{0\}$, take $D=\Z_{pq}$, $C=\{0\}$. Then $C\oplus D=\Z_{pq}$ and $H_D=\supp(|\widehat{\mathbf{1}}_D|^2)=\{0\}=H$.
\item If $H=\Z_{pq}$, take $D=\{0\}$, $C=\Z_{pq}$. Then $C\oplus D=\Z_{pq}$ and $H_D=\Z_{pq}=H$.
\item If $H=\{0\}\cup R_1\cup R_p$, take $D=\{0,1,\ldots,p-1\}$, $C=p\Z_{pq}$. Then $C\oplus D=\Z_{pq}$ and $H_D=\{0\}\cup R_1\cup R_p=H$.
\item If $H=\{0\}\cup R_1\cup R_q$, take $D=\{0,1,\ldots,q-1\}$, $C=q\Z_{pq}$. Then $C\oplus D=\Z_{pq}$ and $H_D=\{0\}\cup R_1\cup R_q=H$.
\item If $H=\{0\}\cup R_q$, take $D=p\Z_{pq}$, $C=\{0,1,\ldots,p-1\}$. Then $C\oplus D=\Z_{pq}$ and $H_D=\{0\}\cup R_q=H$.
\item If $H=\{0\}\cup R_p$, take $D=q\Z_{pq}$, $C=\{0,1,\ldots,q-1\}$. Then $C\oplus D=\Z_{pq}$ and $H_D=\{0\}\cup R_p=H$.
\end{itemize}
By Proposition \ref{prop:delsarte}, for each case we have $D^+(H')=|D|$. Hence on $\Z_{pq}$, $D^+(H')$ can only take the values
\[
    D^+(H')\in\{1,p,q,pq\},
\]
and specifically
\begin{enumerate}
\item $H'=\{0\}\Rightarrow D^+(H')=1$;
\item $H'=\{0\}\cup R_q$ or $H'=\{0\}\cup R_1\cup R_q\Rightarrow D^+(H')=p$;
\item $H'=\{0\}\cup R_p$ or $H'=\{0\}\cup R_1\cup R_p\Rightarrow D^+(H')=q$;
\item $H'=\Z_{pq}\Rightarrow D^+(H')=pq$.
\end{enumerate}

\textbf{Step 3:} Use Delsarte parameters to determine $|A|$.

On one hand, define
\[
    u(\xi)=\frac{|\widehat{\mathbf{1}}_A(\xi)|^2}{|A|^2}.
\]
Then $u(0)=1$, $\supp(u)=H$, $u|_{H^c}\leq 0$, and
\[
    \widehat{u}(x)=\frac{pq}{|A|^2}(\mathbf{1}_A*\mathbf{1}_{-A})(x)\geq 0.
\]
Thus, after possibly averaging to a step function as in \cite{KissLondnerMatolcsiSomlaiFunctional}, which does not change the sum and feasibility, $u$ is a feasible solution for $D^-(H)$ and
\[
    D^-(H)\geq \sum_{\xi\in\Z_{pq}}u(\xi)=\frac{1}{|A|^2}\sum_{\xi}|\widehat{\mathbf{1}}_A(\xi)|^2=\frac{pq}{|A|}.
\]
On the other hand, define
\[
    v(\xi)=\frac{\widehat{f}(\xi)}{\widehat{f}(0)}.
\]
Since $\supp(\widehat{f})\subset H'$, we have $v|_{H^c}=0$ and $v(0)=1$. Moreover,
\[
    \widehat{v}(x)=\frac{pq}{\widehat{f}(0)}f(-x)\geq 0,
\]
since $f\geq 0$. Hence $v$ is a feasible solution for $D^+(H')$ and
\[
    D^+(H')\geq \sum_{\xi}v(\xi)=\frac{1}{\widehat{f}(0)}\sum_{\xi}\widehat{f}(\xi)=\frac{pq\cdot f(0)}{\widehat{f}(0)}=\frac{pq}{|F|}=|A|,
\]
where we used that $\widehat{f}(0)=\sum_x f(x)=:|F|$ and from the convolution equation $\mathbf{1}_A*f=\mathbf{1}_{\Z_{pq}}$ we have $|A|\cdot |F|=pq$, so $|F|=pq/|A|$. Then $\frac{pq}{|F|}=|A|$.

Combining with the duality relation (\ref{eq:duality}),
\[
    |A|=\frac{pq}{pq/|A|}\geq \frac{pq}{D^-(H)}=D^+(H')\geq |A|.
\]
Thus equality holds, and we obtain
\[
    |A|=D^+(H')=|D|\in\{1,p,q,pq\}.
\]

\textbf{Step 4:} Verify the C-M conditions.

Since $\supp(\widehat{\mathbf{1}}_D)=\supp(\widehat{\mathbf{1}}_A)$, equation (\ref{eq:cyclo}) implies that the cyclotomic polynomials $\Phi_p$ and $\Phi_q$ divide the polynomials $A(X)$ and $D(X)$ in exactly the same way. Hence $S_A=S_D$. Moreover, since $C\oplus D=\Z_{pq}$ is a tiling, the set $D$ satisfies the Coven-Meyerowitz condition (T1), and therefore $A$ also satisfies (T1).

It remains to verify (T2). Recall
\[
    S_A:=\{p^\alpha:p^\alpha \text{ is a prime power and } \Phi_{p^\alpha}(X)\mid A(X)\}.
\]
For $M=pq$ with distinct primes, we have $S_A\subset \{p,q\}$.
\begin{itemize}
\item If $|S_A|\leq 1$, i.e., $S_A=\emptyset$, $\{p\}$, or $\{q\}$, then the requirement of (T2), take at least two distinct prime powers, never occurs, so (T2) holds trivially.
\item If $S_A=\{p,q\}$, then by (T1),
\[
    A(1)=\Phi_p(1)\Phi_q(1)=p\cdot q=pq.
\]
Hence $|A|=A(1)=pq=M$. Since $A\subset \Z_{pq}$ and $|\Z_{pq}|=pq$, we must have $A=\Z_{pq}$. Consequently,
\[
    A(X)=1+X+\cdots+X^{pq-1}=\frac{X^{pq}-1}{X-1}.
\]
The cyclotomic factorization $X^{pq}-1=\prod_{d\mid pq}\Phi_d(X)$ implies $\Phi_{pq}(X)\mid A(X)$. Hence (T2) holds as well.
\end{itemize}
Therefore \(A\) satisfies both Coven--Meyerowitz conditions \((T1)\) and \((T2)\).
Since \(M=pq\) is square-free and \(A\) is regarded as a subset of
\(\mathbb Z_M\), the finite cyclic form of the Coven--Meyerowitz criterion
implies that \(A\) tiles \(\mathbb Z_M\). Hence there exists a set
\(B\subseteq \mathbb Z_M\) such that
\(
A\oplus B=\mathbb Z_M .
\)
This proves the nontrivial implication. Conversely, if \(A\oplus B=\mathbb Z_M\),
then Lemma~\ref{lem:tile-implies-pd} shows that \(A\) is a positive-definite
weak tile. Consequently, translational tiles and positive-definite weak tiles
coincide in \(\mathbb Z_{pq}\).
\end{proof}

\begin{remark}
The Delsarte parameters are defined on the space of step functions $\calA(M)$. The feasible functions used in the proof may not belong to $\calA(M)$ initially, but we invoke the averaging procedure from \cite{KissLondnerMatolcsiSomlaiFunctional} to replace them by step functions without changing the sum and feasibility. Hence it is sufficient to consider step functions without loss of generality.
\end{remark}

\subsection{Open problems.}
We have proved that for distinct primes $p,q$, every positive-definite weak tile of $\Z_{pq}$ is necessarily a classical tile. For general finite groups $\Z_M$, we propose the following three problems for further investigation.
\begin{enumerate}
\item Can the equivalence of pd tiling and tiling be extended to more general finite groups $\Z_M$, such as $\Z_{p^nq}$ or $\Z_{pqr}$?
\item Is it possible to establish a unified criterion for positive-definite weak tiling on finite groups using mask polynomials, step classes, Delsarte parameters, and the C-M conditions, thereby completely characterizing when a pd-tile is a tile?

\item For a set \(E\subseteq \mathbb Z_M\), let \(\mu\) be a non-negative weak tiling
measure for \(E\), that is,
\(
\mathbf 1_E * \mu=\mathbf 1_{\mathbb Z_M}.
\)
Is it true that \(\mu\) can be represented as a convex combination of measures
associated with classical tiling complements of \(E\)? More precisely, do there
exist finitely many sets \(T_1,\ldots,T_N\subseteq \mathbb Z_M\) and coefficients
\(c_1,\ldots,c_N\ge 0\), with
\(
\sum_{k=1}^N c_k=1,
\)
such that
\(
\mu=\sum_{k=1}^N c_k\delta_{T_k},
\)
where each \(T_k\) satisfies
\(
E\oplus T_k=\mathbb Z_M?
\)
\end{enumerate}

\vspace{85pt}
{\linespread{1}\selectfont

}

\medskip

\newpage 
{\linespread{1}\selectfont\footnotesize
\noindent\hspace*{12pt}{\scshape Mamateli Kadir: School of Mathematics and Statistics \& Research Center of Modern Mathematics and Applications, Kashi University, Kashi 844000, China}\par
\noindent\hspace*{12pt}\textit{E-mail address}: \texttt{mamatili880@163.com}\par
\medskip
\noindent\hspace*{12pt}{\scshape Kaibo Fan: School of Mathematics and Statistics, Central China Normal University, Wuhan 430079, China}\par
\noindent\hspace*{12pt}\textit{E-mail address}: \texttt{756298493@qq.com}\par
}
\end{document}